\documentclass[10pt]{article}
\usepackage[T1]{fontenc}
\usepackage[cp1250]{inputenc}
\usepackage{lmodern}

\usepackage[english]{babel}
\usepackage{latexsym}
\usepackage{amsmath}
\usepackage{indentfirst}
\usepackage{amsthm}
\usepackage{amssymb}
\usepackage{amsfonts}
\usepackage{stackrel}
\usepackage{dsfont}
\usepackage{tikz}
\usepackage{slashbox}
\usepackage[mathscr]{eucal}
\usepackage{authblk}
\usepackage{hyperref}
\usepackage{mathabx}
\usepackage{bookmark}

\newtheorem{thm}{Theorem}
\newtheorem{lemma}[thm]{Lemma}

\newcommand{\reals}{\mathbb{R}}
\newcommand{\naturals}{\mathbb{N}}
\newcommand{\integers}{\mathbb{Z}}
\newcommand{\complex}{\mathbb{C}}
\newcommand{\eps}{\varepsilon}
\newcommand{\hh}{\mathfrak{H}}
\newcommand{\sinc}{\text{sinc}}

\newcommand{\algebra}{\mathcal{A}}
\newcommand{\CC}{\mathcal{COS}}

\begin{document}
\title{Cosine manifestations of the Gelfand transform}
\author{Mateusz Krukowski}
\affil{Institute of Mathematics, \L\'od\'z University of Technology, \\ W\'ol\-cza\'n\-ska 215, \
90-924 \ \L\'od\'z, \ Poland \\ \vspace{0.3cm} e-mail: mateusz.krukowski@p.lodz.pl}
\maketitle

\begin{abstract}
The goal of the paper is to provide a detailed explanation on how the (continuous) cosine transform and the discrete(-time) cosine transform arise naturally as certain manifestations of the celebrated Gelfand transform. We begin with the introduction of the cosine convolution $\star_c$, which can be viewed as an ``arithmetic mean'' of the classical convolution and its ``twin brother'', the anticonvolution. The d'Alembert property of $\star_c$ plays a pivotal role in establishing the bijection between $\Delta(L^1(G),\star_c)$ and the cosine class $\mathcal{COS}(G),$ which turns out to be an open map if $\mathcal{COS}(G)$ is equipped with the topology of uniform convergence on compacta $\tau_{ucc}$. Subsequently, if $G = \reals,\integers, S^1$ or $\integers_n$ we find a relatively simple topological space which is homeomorphic to $\Delta(L^1(G),\star_c).$ Finally, we witness the ``reduction'' of the Gelfand transform to the aforementioned cosine transforms. 
\end{abstract}

\smallskip
\noindent 
\textbf{Keywords : } Gelfand transform, convolution, structure space, cosine class, continuous/discrete cosine transform
\vspace{0.2cm}
\\
\textbf{AMS Mathematics Subject Classification (2010): } 42A38, 43A20, 43A32,	44A15

\section{Introduction}
\label{Chapter:Introduction}

The current introductory section is divided into three parts. The first one centres around the general framework of our work -- we recall such concepts as the structure space of a Banach algebra, the Gelfand transform or the Haar measure of a locally compact group. Since many have already covered these topics at great length, we restrict ourselves to laying out just the ``basic facts'', i.e., those that are vital in comprehending further sections of the paper. More inquisitive Readers are encouraged to study the materials in the bibliography, which are usually referenced in the footnotes. 

Second part of the introductory section aims at delivering the context for our research regarding the cosine convolution and the cosine transforms. We strive to convince the Reader that the subject of the paper belongs to an active field of mathematical study and even ``creeps into'' the neighbouring disciplines such as algorithmics, computational complexity theory, data and image compression or signal processing. Consequently, our paper can be regarded as a theoretical foundation for various applications in engineering and technology. 

Last but not least, the third and final part of the introduction serves as a brief summary of further sections. We lay out the structure of the paper in order to ``paint the big picture'' of our research prior to delving into technical subtleties.

\subsection{What is the Gelfand transform?}

Without further ado we begin with the basic principles of the Banach algebra theory. In his tour de force ``\textit{A Course in Commutative Banach Algebras}''\footnote{See \cite{Kaniuth}, p. 1.} Eberhard Kaniuth defines a \textit{normed algebra} as a normed linear space $(\algebra,\|\cdot\|)$ over the field of complex numbers $\complex,$ which is also an algebra and whose norm is \textit{submultiplicative}, i.e.,
$$\forall_{f,g \in \algebra}\ \|f\cdot g\| \leqslant \|f\|\cdot \|g\|.$$

\noindent
If $(\algebra,\|\cdot\|)$ is complete (i.e., it is a Banach space), we say that it is a \textit{Banach algebra} (a multitude of examples of Banach algebras is provided in Chapter 6 in \cite{Bobrowski}). Given two Banach algebras $\algebra_1$ and $\algebra_2$ we say that a map $\Psi:\algebra_1\longrightarrow \algebra_2$ is a \textit{Banach algebra homomorphism} if it is 
\begin{itemize}
	\item continuous,
	\item $\complex-$linear, and
	\item \textit{multiplicative}, i.e.,
$$\forall_{f,g\in\algebra_1}\ \Psi(f\cdot g) = \Psi(f)\cdot \Psi(g).$$
\end{itemize}

A \textit{structure space} $\Delta(\algebra)$ of a Banach algebra $\algebra$ is the set of all nonzero Banach algebra homomorphisms $m:\algebra \longrightarrow \complex.$\footnote{For a thorough exposition of this concept see \cite{DeitmarEchterhoff}, p. 43 or \cite{Folland}, p. 5 or \cite{Kaniuth}, p. 46.} Although many authors require the structure space to be defined solely for \textit{commutative} Banach algebras, we intentionally stick to the general case, i.e., not necessarily commutative. The price we pay for taking such an approach is that $\Delta(\algebra)$ can be an empty set -- this happens, for instance, if $\algebra$ is the matrix algebra $M_{n\times n}(\complex).$\footnote{In a unital Banach algebra $\algebra$ every $m\in\Delta(\algebra)$ determines a maximal ideal $\ker(m)$ (see Theorem 1.12 in \cite{Folland}, p. 6 or Theorem 2.1.8 in \cite{Kaniuth}, p. 49). Furthermore, if $\ker(m) = \{0\}$ then $\algebra = \complex.$ Consequently, since Proposition 1.4 in \cite{Grillet}, p. 360 states that the matrix algebra $M_{n\times n}(\complex),\ n > 1$ is simple (it has no nonzero, proper ideals), we conclude that $\Delta(M_{n\times n}(\complex)) = \emptyset.$} 

The set $\Delta(\algebra)$ becomes a locally compact space\footnote{Every topological space that appears in the paper is assumed to be Hausdorff so we refrain from writing that explicitly.} when endowed with the weak* topology\footnote{For a detailed exposition of the weak* topology $\tau^*$ see Chapter 3.4 in \cite{Brezis}, p. 62 or Chapter 2.4 in \cite{Pedersen}, p. 62. The fact that $\Delta(\algebra)$ is a locally compact space (under $\tau^*$) can be found as Theorem 2.4.5 in \cite{DeitmarEchterhoff}, p. 46 or Theorem 2.2.3 in \cite{Kaniuth}, p. 52. The former theorem also establishes the existence of the Gelfand transform -- for another point of view see Theorem 1.13 in \cite{Folland}, p. 7.} and its raison d'\^etre lies in the fact that (if $\Delta(\algebra)\neq \emptyset$) there exists a norm-decreasing, algebra homomorphism $\Gamma: \algebra \longrightarrow C_0(\Delta(\algebra))$ given by the formula 
$$\forall_{m\in\Delta(\algebra)}\ \Gamma(f)(m) := m(f).$$

\noindent
The function $\Gamma(f)\in C_0(\Delta(\algebra))$ is commonly referred to as the \textit{Gelfand transform} of the element $f\in\algebra$ and the notation ``$\Gamma(f)$'' is usually reduced to $\widehat{f}.$  

In order to reinforce our intuition regarding the abstract theory above let us focus on a particular instance of a commutative Banach algebra. Let $G$ be a locally compact abelian group and let $\mu$ be its \textit{Haar measure},\footnote{See \cite{Haar} for the original paper by Haar as well as \cite{Cartan, DiestelSpalsbury} or \cite{Weil} (chapter 2) for a further study of the subject.} i.e., a nonzero, Borel measure which is
\begin{itemize}
	\item finite on compact sets,
	\item \textit{inner regular}, i.e., for every open set $U$ in $G$ we have
	$$\mu(U) = \sup \{\mu(K)\ :\ K\subset U,\ K - \text{compact}\},$$
	
	\item \textit{outer regular}, i.e., for every Borel set $A$ in $G$ we have
	$$\mu(A) = \inf \{\mu(U)\ :\ A\subset U,\ U - \text{open}\},$$
	
	\item \textit{translation-invariant}, i.e., for every $x\in G$ and every Borel set $A$ we have $\mu(x+A) = \mu(A).$ 
\end{itemize} 

\noindent
For every two functions $f,g\in L^1(G)$ we define their \textit{convolution} $f\star g$ with the formula 
\begin{gather}
\forall_{x\in G}\ f\star g(x) := \int_G\ f(y)g(x-y)\ dy,
\label{convolution}
\end{gather}

\noindent
where the integration is with respect to the Haar measure $\mu$.\footnote{Formally, we probably should write ``$d\mu(y)$'' when integrating with respect to the Haar measure, yet we feel that this is an unnecessary complication of the nomenclature. Hence, we stick to an abbreviated form ``$dy$'' for the sake of simplicity.} This operation turns $L^1(G)$ into a commutative Banach algebra (obviously $L^1(G)$ is already a Banach space, so $\star$ defines the ``mutliplication'' of the elements). 

It is one of the gems of abstract harmonic analysis that the structure space $\Delta(L^1(G),\star)$ is homeomorphic\footnote{See Theorem 3.2.1 in \cite{DeitmarEchterhoff}, p. 67 or Theorem 2.7.2 in \cite{Kaniuth}, p. 89.} to the dual group $\widehat{G},$ i.e., the group of all nonzero, multiplicative homomorphisms (called \textit{characters}) $\chi: G \longrightarrow S^1.$ The homeomorphism $\hh:\Delta(L^1(G),\star) \longrightarrow \widehat{G}$ in question assigns a unique $\chi\in\widehat{G}$ to every multiplicative linear functional $m\in\Delta(L^1(G),\star)$ in such a way that 
$$\forall_{f\in L^1(G)}\ m(f) = \int_G\ f(x)\overline{\chi(x)}\ dx.$$

\noindent
As a consequence of $\Delta(L^1(G),\star)$ and $\widehat{G}$ being homeomorphic, the Banach algebra $C_0(\Delta(L^1(G),\star))$ is isometrically isomorphic with $C_0(\widehat{G})$.\footnote{See Corollary 2.2.13 in \cite{Kaniuth}, p. 57 for a more general result stating that Banach algebras $C_0(X)$ and $C_0(Y)$ are isometrically isomorphic if and only if locally compact Hausdorff spaces $X$ and $Y$ are homeomorphic.} It follows that we may treat the Gelfand transform $\widehat{f}\in C_0(\Delta(L^1(G),\star))$ of $f\in L^1(G)$ as an element of $C_0(\widehat{G})$ given by the formula
$$\forall_{f\in L^1(G)}\ \widehat{f}(\chi) = \int_G\ f(x)\overline{\chi(x)}\ dx.$$

\noindent
To put it in other words, the Gelfand transform of $f\in L^1(G)$ is the Fourier transform $\widehat{f} \in C_0(\widehat{G}).$

\subsection{Context}

The aim of the current subsection is to provide the context for the upcoming sections by giving a brief overview of the present state of the literature. We strive to substantiate the claim that the cosine transform is not only an invaluable tool in the mathematical toolbox (especially for those working in the field of harmonic analysis), but its usage stretches outside of the mathematical realm, to fields such as engineering and technology. 

Although the concept of the cosine transform goes as far back as the monumental work ``\textit{The analytical theory of heat}'' by Jean Baptiste Fourier,\footnote{For an english translation of Fourier's book with commentary by Alexander Freeman see \cite{Fourier}.} it is widely agreed that the ``modern history'' of the operator starts roughly 50 years ago. As Nasir Ahmed recollects,\footnote{See \cite{Ahmeddiary}.} in late '60s and early '70s ``there was a great deal of research activity related to digital orthogonal transforms'' and a large number of transforms were introduced ``with claims of better performance relative to others transforms''. Inspired by the Karhunen-Loeve transform, Ahmed came up with the idea of the cosine transform and issued a proposal to the National Science Foundation to study the newly-invented operator. His proposal was dismissed due to the idea being ``too simple'', but Ahmed continued to work on the concept with his Ph.D. student T. Natarajan and a colleague at the University of Texas, Dr. K. R. Rao. The team was quickly surprised how well the cosine transform performed relative to other transforms and published their results in 1974.\footnote{See \cite{AhmedNatarajanRao}.} 17 years later, Ahmed wrote ``Little did we realize at that time that the resulting DCT\footnote{DCT stands for ``discrete cosine transform''.} would be widely used in the future! It is indeed gratifying to see that the DCT is now essentially a standard in the area of image data compression via transform coding techniques''.

A fleeting glimpse at the literature reveals that Ahmed's claim is not unsubstantiated. The cosine transform has found applications in:
\begin{itemize}
	\item algorithmics and computational complexity theory (\cite{AkopiaAstolaSaarinenTakala, ArguelloZapata, ChenFralickSmith, GuoShiWang, Hou, Lee, Li, MillerTseng, NussbaumerVetterli, PlonkaTasche, Wang}),
	\item data compression (see \cite{Hoffman}),
	\item image compression and JPEG format (see \cite{AguiAraiNakajima, AhmedMagotraMandyam, Wallace}),
	\item signal processing (see \cite{JohnsonShao}).
\end{itemize}

\noindent
We believe that such broad interest in the subject of cosine transform is not coincidental and reflects the fact that the topic is still within the scope of active mathematical research.

\subsection{Layout of the paper}

In this final part of the introduction we summarize the structure of the paper. The goal of this quick outline is to facilitate the comprehension of the ``big picture'' before delving into technical subtleties. 

Section \ref{section:cosineconvolution} begins with the definition of the cosine convolution and what follows is an investigation of its basic features like the d'Alembert property (see Theorem \ref{dAlembertproperty}). We proceed with establishing a bijection $\beta_G$ between the structure space $\Delta(L^1(G), \star_c)$ and the cosine class $\CC(G)$ (see Theorem \ref{cosinebijection}). 

Section \ref{section:structurespacesandcosineclasses} adds a topological layer to our analysis -- by topologizing the set $\CC(G)$ with the topology of uniform convergence on compacta we discover that $\beta_G$ is in fact an open map. The section goes on to prove that $\CC(G)$ is homeomorphic to
\begin{itemize}
	\item $\reals_+\cup\{0\}$ if $G = \reals,$
	\item $S^1_+\cup\{1\}$ if $G = \integers,$ where 
	$$S^1_+ := \bigg\{z \in S^1\ :\ \text{Im}(z) > 0\bigg\},$$
	\item $\naturals_0$ if $G = S^1,$
	\item $\integers_{\lceil \frac{n+1}{2}\rceil}$ if $G = \integers_n$ ($x\mapsto\lceil x\rceil$ is the ceiling function).
\end{itemize}

The focal point of Section \ref{section:cosinestructurespaces} is the computation of particular instances of (what we call) \textit{cosine structure spaces}. It turns out (see Theorems \ref{deltal1reals}, \ref{deltaell1Z}, \ref{deltal1realsslashintegers} and \ref{deltaell1}) that if $G = \reals, \integers, S^1$ or $\integers_n$ then $\beta_G : \Delta(L^1(G),\star_c) \longrightarrow \CC(G)$ is continuous (and thus a homeomorphism in view of what we said earlier). The section concludes with bringing all the pieces of the puzzle as we witness the Gelfand transform manifest itself in the form of the cosine transforms.  

Our final remarks take the form of an Epilogue, in which we indicate a possible direction of future research. Naturally, the paper concludes with a bibliography.

\section{Cosine convolution}
\label{section:cosineconvolution}

We begin our story with the definition of the \textit{cosine convolution} operator $\star_c:L^1(G)\times L^1(G)\longrightarrow \complex$ with the formula:
\begin{gather}
\forall_{x\in G}\ f\star_c g(x) := \int_G\ f(y)\cdot \frac{g(x+y) + g(x-y)}{2}\ dy.
\label{cosineconvolution}
\end{gather}

\noindent
The question that immediately springs to mind when looking at formula \eqref{cosineconvolution} is whether the cosine convolution is well-defined? In other words, given two functions $f,g \in L^1(G)$ does $f\star_c g(x)$ make sense for at least some values $x\in G$? The answer is (luckily) affirmative and in fact, $f\star_c g$ defines a multiplication on $L^1(G)$ turning it into a Banach algebra:  

\begin{thm}
If $f,g\in L^1(G)$ then $f\star_c g(x)$ is well-defined (i.e., finite for almost every $x \in G$) and we have 
$$\|f\star_c g\|_1 \leqslant \|f\|_1\|g\|_1.$$

\noindent
Consequently, $(L^1(G),\star_c)$ is a Banach algebra. 
\label{cosineBanachalgebra}
\end{thm}

We have taken the liberty of omitting the proof of this result as it is almost a verbatim rewrite of the proof of Theorem 1.6.2 in \cite{Deitmar}, p. 26 where Deitmar shows (with painstaking precision) that $L^1(G)$ is a Banach algebra under the standard convolution $\star$ defined by \eqref{convolution}. Whoever shall read Deitmar's reasoning will surely discover that his proof works equally well for the \textit{anticonvolution}:
\begin{gather*}
\forall_{x\in G}\ f\star_a g(x) := \int_G\ f(y)g(x+y)\ dy.
\end{gather*}

\noindent
Thus, if we view the cosine convolution $\star_c$ as an ``arithmetic mean of convolutions'', i.e., $\frac{\star + \star_a}{2}$ then Theorem \ref{cosineBanachalgebra} becomes trivial. There is one more issue, however, that we would like to address at this point. Although $(L^1(G),\star_c)$ is a Banach algebra, it need not be a commutative Banach algebra, as the anticonvolution part need not be commutative:
\begin{equation*}
\begin{split}
\forall_{x\in G}\ f\star_a g(x) &= \int_G\ f(y)g(x+y)\ dy \stackrel{y\mapsto -(y+x)}{=} \int_G\ f(-(y+x))g(-y)\ dy\\
&\stackrel{y\mapsto -y}{=} \int_G\ g(y)f(-x+y)\ dy = g\star_a f(-x).
\end{split}
\end{equation*}

\noindent
In spite of this nagging splinter, we stick to our guns by defining $\Delta(L^1(G),\star_c)$ as the set of all nonzero Banach algebra homomorphisms $m: (L^1(G),\star_c) \longrightarrow \complex.$ 

Amongst all properties of the cosine convolution we focus on the one bearing a close resemblance with the classical d'Alembert functional equation
$$\forall_{x,y\in \reals}\ \phi(x)\phi(y) = \frac{\phi(x+y) + \phi(x-y)}{2},$$

\noindent
whose continuous and bounded (and nonzero) solutions are the functions $x\mapsto \cos(\omega x),\ \omega\in \reals.$\footnote{For completeness we should mention that if we allow for unbounded (yet still continuous) solutions then the family of functions $x\mapsto \cosh(\omega x),\ \omega \in \reals$ also satisfies the classical d'Alembert functional equation. What is more, apart from the zero function and the two families $x\mapsto \cos(\omega x),\ x\mapsto \cosh(\omega x),\ \omega\in\reals$ there are no other solutions.} After all, it is hardly surprising that the cosine convolution should have so much in common with the cosine function itself! Without further ado, here is the d'Alembert property of the cosine convolution:

\begin{thm}
Let $x,y\in G.$ If $g\in L^1(G)$ is an even function, then 
\begin{gather}
L_yg \star_c L_xg = g\star_c \frac{L_{x+y}g + L_{x-y}g}{2},
\label{LxgstarcLyg}
\end{gather}

\noindent
where for every $z\in G$ the operator $L_z:L^1(G)\longrightarrow L^1(G)$ is given by 
\begin{gather*}
\forall_{u\in G}\ L_zf(u) := f(u - z).
\label{shiftoperators}
\end{gather*}
\label{dAlembertproperty}
\end{thm}
\begin{proof}
First we note that for every $u \in G$ we have
\begin{equation}
\begin{split}
&\int_G\ L_yg(v) L_xg(u+v)\ dv \stackrel{v\mapsto v+y}{=} \int_G\ g(v) L_{x-y}g(u+v)\ dv,\\
&\int_G\ L_yg(v) L_xg(u+v)\ dv \stackrel{v\mapsto -v+y}{=} \int_G\ g(-v) L_{x-y}g(u-v)\ dv = \int_G\ g(v) L_{x-y}g(u-v)\ dv,
\end{split}
\label{eqn1forcosine}
\end{equation}
	
\noindent
and analogously
\begin{equation}
\begin{split}
&\int_G\ L_yg(v) L_xg(u-v)\ dv \stackrel{v\mapsto v+y}{=} \int_G\ g(v) L_{x+y}g(u-v)\ dv,\\
&\int_G\ L_yg(v) L_xg(u-v)\ dv \stackrel{v\mapsto -v+y}{=} \int_G\ g(-v) L_{x+y}g(u+v)\ dv = \int_G\ g(v) L_{x+y}g(u+v)\ dv.
\end{split}
\label{eqn2forcosine}
\end{equation}
	
\noindent
Summation of equations \eqref{eqn1forcosine} (and division by $2$) yields
\begin{gather}
\forall_{u\in G}\ \int_G\ L_yg(v) L_xg(u+v)\ dv = \int_G\ g(v)\cdot \frac{L_{x-y}g(u+v) + L_{x-y}g(u-v)}{2}\ dv = g\star_c L_{x-y}g(u),
\label{eqn3forcosine}
\end{gather}
	
\noindent	
while the summation of equations \eqref{eqn2forcosine} (and again division by $2$) yields
\begin{gather}
\forall_{u\in G}\ \int_G\ L_yg(v) L_xg(u-v)\ dv = \int_G\ g(v)\cdot \frac{L_{x+y}g(u+v) + L_{x+y}g(u-v)}{2}\ dv = g\star_c L_{x+y}g(u).
\label{eqn4forcosine}
\end{gather}

\noindent
Finally, summation of equations \eqref{eqn3forcosine} and \eqref{eqn4forcosine}	(and division by $2$ for the last time) produces the desired d'Alembert property \eqref{LxgstarcLyg}.
\end{proof}

With the d'Alembert property of the cosine convolution in our toolbox we begin to investigate the structure space $\Delta(L^1(G),\star_c).$ Our first goal is to prove that $\Delta(L^1(G),\star_c)$ is in bijective correspondence with the \textit{cosine class} 
\begin{gather*}
\CC(G) := \bigg\{\phi \in C^b(G)\ :\ \phi\neq 0,\ \forall_{x,y\in G}\ \phi(x)\phi(y) = \frac{\phi(x+y) + \phi(x-y)}{2}\bigg\}.
\end{gather*}

\noindent
In order to demonstrate this bijection, we will use the following lemma:

\begin{lemma}
For every $m\in \Delta(L^1(G),\star_c)$ there exists an even function $g_*\in L^1(G)$ such that $m(g_*)\neq 0.$
\label{existenceofevenfunction}
\end{lemma}
\begin{proof}
Let $\iota:G\longrightarrow G$ be the \textit{inverse function} $\iota(x) := -x.$ Then
\begin{equation*}
\begin{split}
\forall_{f\in L^1(G)}\ (f\circ\iota) \star_c f(x) &= \int_G\ f\circ\iota(y) \cdot \frac{f(x+y) + f(x-y)}{2}\ dy\\
&\stackrel{y\mapsto -y}{=} \int_G\ f(y) \cdot \frac{f(x-y) + f(x+y)}{2}\ dy = f\star_c f,
\end{split}
\end{equation*}

\noindent
which leads to 
\begin{gather}
\forall_{f\in L^1(G)}\ m(f\circ\iota)m(f) = m((f\circ\iota) \star_c f) = m(f\star_c f) = m(f)^2.
\label{canwedividebymf}
\end{gather}

\noindent
Since $m$ is nonzero by the definition of $\Delta(L^1(G),\star_c),$ then there must exist a function $f_*\in L^1(G)$ such that $m(f_*)\neq 0.$ Consequently, equations \eqref{canwedividebymf} yield $m(f_*\circ\iota) = m(f_*).$ Finally, we put $g_* := f + f\circ\iota$ and observe that 
$$m(g_*) = m(f_*+f_*\circ\iota) = m(f_*) + m(f_*\circ\iota) = 2m(f_*) \neq 0,$$

\noindent
which concludes the proof.  
\end{proof}

We are now in position to prove a bijective correspondence between $\Delta(L^1(G),\star_c)$ and $\CC(G):$

\begin{thm}
If $m_{\phi} : L^1(G) \longrightarrow \complex$ is given by
\begin{gather}
m_{\phi}(f) := \int_G\ f(x)\phi(x)\ dx
\label{formulaformphi}
\end{gather}

\noindent
for some $\phi\in \CC(G),$ then $m_{\phi} \in \Delta(L^1(G),\star_c).$ Furthermore, for every $m \in \Delta(L^1(G),\star_c)$ there exists a unique $\phi\in \CC(G)$ such that $m = m_{\phi}.$
\label{cosinebijection}
\end{thm}
\begin{proof}
Obviously, if $m_{\phi}$ is given by formula (\ref{formulaformphi}) then it is a linear functional, which satisfies
$$\forall_{f\in L^1(G)}\ |m_{\phi}(f)| \leqslant \|\phi\|_{\infty}\|f\|_1,$$

\noindent
where $\|\cdot\|_{\infty}$ stands for the supremum norm. The inequality above means that $m_{\phi}$ is a bounded (and thus continuous) functional on $L^1(G)$. Thus, in order to conclude that $m_{\phi}\in \Delta(L^1(G),\star_c)$ we simply need to demonstrate the muliplicativity of $m_{\phi}.$ We have 
\begin{equation*}
\begin{split}
\forall_{f,g\in L^1(G)}\ m_{\phi}(f\star_c g) &= \int_G\ f\star_c g(x)\phi(x)\ dx \\
&= \int_G \bigg(\int_G\ f(y)\cdot \frac{g(x+y) + g(x-y)}{2}\ dy\bigg)\phi(x)\ dx \\
&=  \int_G \int_G\ f(y)\cdot \frac{g(x+y) + g(x-y)}{2}\cdot \phi(x)\ dxdy \\
&= \frac{1}{2} \int_G \int_G\ f(y)g(x+y)\phi(x)\ dxdy + \frac{1}{2}\int_G \int_G\ f(y) g(x-y) \phi(x)\ dxdy.
\end{split}
\end{equation*}

\noindent
Applying the substition $x\mapsto x-y$ to the first double integral and the substitution $x\mapsto x+y$ to the second double integral in the last line we obtain
\begin{equation*}
\begin{split}
\forall_{f,g\in L^1(G)}\ m_{\phi}(f\star_c g) &= \int_G\int_G\ f(y)g(x)\cdot \frac{\phi(x-y) + \phi(x+y)}{2}\ dxdy\\
&\stackrel{\phi\in \CC(G)}{=} \int_G\int_G\ f(y)g(x)\phi(x)\phi(y)\ dxdy = m_{\phi}(f)m_{\phi}(g).
\end{split}
\end{equation*}

\noindent
This implies that $m_{\phi} \in \Delta(L^1(G), \star_c)$ and concludes the first part of the proof.

For the second part of the proof, we fix an element $m\in \Delta(L^1(G),\star_c)$ and seek to show that there exists a unique $\phi\in\CC(G)$ such that $m = m_{\phi}.$ By Lemma \ref{existenceofevenfunction} there exists an even function $g_* \in L^1(G)$ such that $m(g_*)\neq 0.$ Moreover, by Theorem \ref{dAlembertproperty} we know that 
$$\forall_{x,y\in G}\ L_yg_*\star_c L_xg_* = g_*\star_c \frac{L_{x+y}g_* + L_{x-y}g_*}{2}.$$

\noindent
Applying the functional $m$ to this equation and using its linearity and multiplicativity we obtain
\begin{gather}
\forall_{x,y\in G}\ m(L_yg_*) m(L_xg_*) = m(g_*)\cdot \frac{m(L_{x+y}g_*) + m(L_{x-y}g_*)}{2}.
\label{funceqnforcos}
\end{gather}

\noindent
Next, we define the function $\phi : G\longrightarrow \complex$ by the formula
\begin{gather}
\forall_{x\in G}\ \phi(x) := \frac{m(L_xg_*)}{m(g_*)},
\label{definitionofphi}
\end{gather}

\noindent
which is 
\begin{itemize}
	\item nonzero, because $\phi(0) = 1,$
	\item continuous by Lemma 1.4.2 in \cite{DeitmarEchterhoff}, p. 18, and
	\item bounded, because $\|L_xg_*\|_1 = \|g_*\|_1$ for every $x \in G.$
\end{itemize}

\noindent
Dividing equation (\ref{funceqnforcos}) by $m(g_*)^2$ we may rewrite it in the form 
$$\forall_{x,y\in G}\ \phi(y) \phi(x) =  \frac{\phi(x+y) + \phi(x-y)}{2},$$

\noindent
which means that $\phi\in \CC(G).$ We have
\begin{equation*}
\begin{split}
\forall_{f\in L^1(G)}\ \int_G\ f(x)\phi(x)\ dx &\stackrel{\phi \in \CC(G)}{=} \int_G\ f(x)\cdot \frac{\phi(x) + \phi(-x)}{2}\ dx\\
&\stackrel{\eqref{definitionofphi}}{=} \int_G\ f(x)\cdot \frac{m(L_xg_*) + m(L_{-x}g_*)}{2m(g_*)}\ dx \\
&= \frac{1}{m(g_*)}\cdot m\left(\int_G\ f(x)\cdot \frac{L_xg_* + L_{-x}g_*}{2}\ dx \right) \\
&= \frac{1}{m(g_*)}\cdot m(f\star_c g_*) = \frac{1}{m(g_*)}\cdot m(f)m(g_*) = m(f),
\end{split}
\end{equation*}

\noindent
where the third equality holds true due to Lemma 11.45 in \cite{AliprantisBorder}, p. 427 (or Proposition 7 in \cite{Dinculeanu}, p. 123). We have thus proved that $m = m_{\phi}.$ 

Finally, the fact that \eqref{definitionofphi} is a unique $\phi$ such that $m=m_{\phi}$ follows from a technique, which is well-known in the field of variational calculus: suppose that there exist $\phi_1,\phi_2 \in \CC(G)$ such that $m = m_{\phi_1} = m_{\phi_2}.$ Consequently, we have 
\begin{gather}
\forall_{f\in L^1(G)}\ \int_G\ f(x)\bigg(\phi_1(x) - \phi_2(x)\bigg)\ dx = 0.
\label{duBoiszero}
\end{gather}

\noindent
Assuming the existence of an element $x_*\in G$ such that $\phi_1(x_*)\neq \phi_2(x_*),$ we can choose an open neighbourhood $U_*$ of $x_*$ such that 
\begin{itemize}
	\item $\phi_1 - \phi_2$ is of constant sign on $U_*,$ and
	\item $\overline{U_*}$ is compact (because we constantly work under the assumption that the group $G$ is locally compact). 
\end{itemize}

\noindent
It remains to observe that for $f_* := (\phi_1 - \phi_2)\mathds{1}_{\overline{U_*}}$ we have
$$0 \stackrel{\eqref{duBoiszero}}{=} \int_G\ f_*(x)\bigg(\phi_1(x) - \phi_2(x)\bigg)\ dx = \int_{\overline{U_*}}\ \bigg(\phi_1(x) - \phi_2(x)\bigg)^2\ dx > 0,$$

\noindent
which is a contradiction. This means that $\phi_1 = \phi_2$ and concludes the proof.
\end{proof}

\section{Structure spaces and cosine classes}
\label{section:structurespacesandcosineclasses}

To summarize the climactic point of the previous section, Theorem \ref{cosinebijection} establishes that the function $\beta_G: \Delta(L^1(G),\star_c) \longrightarrow \CC(G)$ given by $\beta_G(m) := \phi$ (where $\phi$ is a unique element of $\CC(G)$ such that $m=m_{\phi}$) is a bijection. Our next task is to prove that $\beta_G$ is in fact more than just a bijection of two sets. Before we discuss the topology that we impose on the cosine class $\CC(G)$ let us examine its elements a little closer:

\begin{thm}
For every $\phi\in\CC(G)$ there exists a character $\chi_{\phi}\in\widehat{G}$ such that 
\begin{gather}
\forall_{x\in G}\ \phi(x) = \frac{\chi_{\phi}(x) + \overline{\chi_{\phi}(x)}}{2}.
\label{formofphi}
\end{gather}

\noindent
In particular, $\|\phi\|_{\infty} \leqslant 1.$
\label{cosineclassanddual}
\end{thm}
\begin{proof}
By Corollary 1 in \cite{Kannappan} there exists a continuous homomorphism $\chi_{\phi}:G\longrightarrow \complex^*$ such that 
\begin{gather}
\forall_{x\in G}\ \phi(x) = \frac{\chi_{\phi}(x) + \chi_{\phi}(x)^{-1}}{2}.
\label{formofphiinverse}
\end{gather}

\noindent
Furthermore, since $\phi$ is bounded, then so is $\chi_{\phi}.$ We fix $y_*\in G$ and observe that the muliplicative property 
$$\forall_{x\in G}\ \chi_{\phi}(x) \chi_{\phi}(y_*) = \chi_{\phi}(x+y_*)$$

\noindent
implies 
$$\sup_{x\in G}\ |\chi_{\phi}(x)| |\chi_{\phi}(y_*)| = \sup_{x\in G}\ |\chi_{\phi}(x+y_*)| = \sup_{z\in G}\ |\chi_{\phi}(z)|.$$

\noindent
By the finiteness of $\sup_{x\in G}\ |\chi_{\phi}(x)| = \sup_{z\in G}\ |\chi_{\phi}(z)|$ we obtain $|\chi_{\phi}(y_*)| = 1.$ Since the element $y_*$ was chosen arbitrarily, then we have established that $\chi_{\phi}$ is in fact a character, i.e., $\chi_{\phi} \in \widehat{G}.$ Formula \eqref{formofphi} follows from \eqref{formofphiinverse} and the fact that $\chi_{\phi}(x)^{-1} = \overline{\chi_{\phi}(x)}$ for every $x\in G$. 
\end{proof}

Knowing what the elements of $\CC(G)$ look like, we intend to show that $\beta_G$ is an open map if the cosine class is endowed with the proper topology.\footnote{As we have already mentioned in the introductory section, $\Delta(L^1(G),\star_c)$ is endowed with the weak* topology $\tau^*$, which makes it a locally compact space.} What do we mean by ``proper''? Well, $\CC(G)$ is a subspace of $C^b(G)$ so it is natural to endow it with the topology of uniform convergence on compacta $\tau_{ucc}$.\footnote{For a detailed discussion on the properties of $\tau_{ucc}$ see Chapter 7 in \cite{Kelley}, Chapter 46 in \cite{Munkres} or Chapter 43 in \cite{Willard}.} Our next result confirms that this is the right choice:

\begin{thm}
$\beta_G : (\Delta(L^1(G),\star_c), \tau^*)\longrightarrow (\CC(G),\tau_{ucc})$ is an open map. 
\label{betaGisanopenmap}
\end{thm}
\begin{proof}
Our task is to prove that the image (under $\beta_G$) of an arbitrary set 
\begin{gather*}
U := \bigg\{m\in \Delta(L^1(G),\star_c) \ :\ \forall_{n=1,\ldots,N}\ |m(f_n) - m_{\phi_*}(f_n)| < \eps\bigg\},
\label{weakstaropenset}
\end{gather*}

\noindent
where $\eps>0,\ m_{\phi_*} \in \Delta(L^1(G),\star_c)$ and $(f_n)_{n=1}^N\subset L^1(G)$ are fixed, is $\tau_{ucc}-$open. We fix $\phi_{**}\in \beta_G(U),$ which means that there exists $\delta\in(0,1)$ such that 
\begin{gather}
\forall_{n=1,\ldots,N}\ |m_{\phi_{**}}(f_n) - m_{\phi_*}(f_n)| < \delta\eps.
\label{deltaeps}
\end{gather}

\noindent
Furthermore, we pick $K$ to be a compact subset of $G$ such that 
\begin{gather}
\forall_{n=1,\ldots,N}\ \int_{G\backslash K}\ |f_n(x)|\ dx \leqslant \frac{(1-\delta)\eps}{4}.
\label{choiceofK}
\end{gather}

\noindent
Last but not least, we put
\begin{gather*}
V := \bigg\{\phi\in \CC(G)\ :\ \sup_{x\in K}\ |\phi(x) - \phi_{**}(x)| < \frac{(1-\delta)\eps}{2\max_{n=1,\ldots,N}\ \|f_n\|_1}\bigg\},
\label{Vopenintucc}
\end{gather*}

\noindent
which is a $\tau_{ucc}-$open neighbourhood of $\phi_{**}.$ Finally, we calculate that 
\begin{equation*}
\begin{split}
\forall_{\phi\in V}\ \forall_{n=1,\ldots,N}\ |m_{\phi}(f_n) - m_{\phi_*}(f_n)| &\leqslant |m_{\phi}(f_n) - m_{\phi_{**}}(f_n)| + |m_{\phi_{**}}(f_n) - m_{\phi_*}(f_n)|\\
&\stackrel{\eqref{deltaeps}}{\leqslant} |m_{\phi}(f_n) - m_{\phi_{**}}(f_n)| + \delta\eps \\
&\leqslant \int_K\ |f_n(x)| |\phi(x) - \phi_{**}(x)|\ dx + \int_{G\backslash K}\ |f_n(x)| |\phi(x) - \phi_{**}(x)|\ dx + \delta\eps\\
&\stackrel{\text{Lemma }\ref{cosineclassanddual}}{\leqslant} \int_K\ |f_n(x)| |\phi(x) - \phi_{**}(x)|\ dx + 2 \int_{G\backslash K}\ |f_n(x)|\ dx + \delta\eps\\
&\stackrel{\eqref{choiceofK}}{\leqslant} \int_K\ |f_n(x)| |\phi(x) - \phi_{**}(x)|\ dx + \frac{(1+\delta)\eps}{2}\\
&\stackrel{\phi\in V}{<} \frac{(1-\delta)\eps}{2\max_{n=1,\ldots,N}\ \|f_n\|_1}\cdot \int_K\ |f_n(x)| \ dx + \frac{(1+\delta)\eps}{2} \leqslant \eps.
\end{split}
\end{equation*}

\noindent
We conclude that $V$ is a $\tau_{ucc}-$open neighbourhood of (an arbitrarily chosen) $\phi_{**}$ and $V \subset \beta_G(U).$ Thus $\beta_G$ is an open map.
\end{proof}

The question of continuity of $\beta_G$ is more subtle. In fact, we do not know whether this function is continuous for an arbitrary locally compact abelian group $G$, but it turns out to be true for very important particular cases. We will come back to this issue in the next section. However, before we do that we wish to investigate the cosine classes a little further. 

Our goal is to ``compute''\footnote{By ``computing'' the cosine class $\CC(G)$ we mean ``finding a (relatively simple) topological space $T$ which is homeomorphic to $\CC(G)$''.} the cosine classes $\CC(G)$ if $G = \reals,\integers,S^1$ and $\integers_n.$ We will refer to these families as the \textit{canonical cosine classes} since the four groups $\reals,\integers,S^1$ and $\integers_n$ play a fundamental role in commutative harmonic analysis. 

It is well-known in the literature\footnote{See Proposition 7.1.6 in \cite{Deitmar}, p. 106 or Theorem 4.5 in \cite{Folland}, p. 89.} that
\begin{equation*}
\begin{split}
\widehat{\reals} &= \bigg\{x\mapsto e^{2\pi iyx} \ :\ y\in\reals\bigg\}, \\
\widehat{\integers} &= \bigg\{k\mapsto z^k\ :\ z \in S^1\bigg\},\\
\widehat{S^1} &= \bigg\{x \mapsto e^{2\pi ikx} \ :\ k\in\integers\bigg\},\\
\widehat{\integers_n} &= \bigg\{k\mapsto e^{\frac{2\pi ilk}{n}} \ :\ l\in\integers_n\bigg\},\\
\end{split}
\end{equation*}
 
\noindent
so by Theorem \ref{cosineclassanddual} we have
\begin{equation*}
\begin{split}
\CC(\reals) &= \bigg\{x\mapsto \frac{e^{2\pi iyx} + e^{-2\pi iyx}}{2} = \cos(2\pi yx) \ :\ y\in\reals_+\cup\{0\}\bigg\},\\
\CC(\integers) &= \bigg\{k\mapsto \frac{z^k + z^{-k}}{2}\ :\ z\in S^1_+\cup\{1\}\bigg\},\ \text{ where }\ S^1_+ := \bigg\{z \in S^1\ :\ \text{Im}(z) > 0\bigg\},\\
\CC(S^1) &= \bigg\{x\mapsto \frac{e^{2\pi ikx} + e^{-2\pi ikx}}{2} = \cos(2\pi kx) \ :\ k\in\naturals_0\bigg\},\\
\CC(\integers_n) &= \bigg\{k\mapsto \frac{e^{\frac{2\pi ilk}{n}} + e^{-\frac{2\pi ilk}{n}}}{2} = \cos\left(\frac{2\pi lk}{n}\right) \ :\ l\in\integers_{\lceil\frac{n+1}{2}\rceil}\bigg\}.
\end{split}
\end{equation*}

\noindent
We go on to prove that $(\CC(\reals),\tau_{ucc})$ is homeomorphic to $\reals_+\cup\{0\}$ but first we need the following technical lemma:

\begin{lemma}
If $y_* \in \reals_+\cup\{0\}$ and $(y_n)\subset \reals_+\cup\{0\}$ is an unbounded sequence, then for every $\eps>0$ the sequence
$$n\mapsto \sup_{x\in [0,\eps]}\ |\cos(2\pi y_nx) - \cos(2\pi y_*x)|$$

\noindent
does not converge to zero.
\label{lemmaonnonconvergence}
\end{lemma}
\begin{proof}
If $y_* = 0$ then for any $y_n\geqslant 1$ the function $x\mapsto \cos(2\pi y_nx)$ has period $T_n = \frac{1}{y_n} \leqslant 1$ so 
$$\forall_{y_n\geqslant 1}\ \sup_{x\in [0,1]}\ |\cos(2\pi y_nx) - 1| \geqslant 2,$$

\noindent
which ends the proof. Therefore, suppose that $y_* \neq 0$ and put $K:= [0,\frac{1}{8y_*}].$ Then
$$\forall_{x\in K\cap [0,\eps]}\ \cos(2\pi y_* x) \geqslant \frac{\sqrt{2}}{2}$$

\noindent
whereas for any $y_n \geqslant \max\left(4y_*,\frac{1}{2\eps}\right)$ the function $x\mapsto \cos(2\pi y_n x)$ has period $T_n = \frac{1}{y_n}\leqslant \min(\frac{1}{4y_*},2\eps)$ and thus attains the value $-1$ on $K\cap[0,\eps]$. Consequently, we have 
$$\forall_{y_n\geqslant \max\left(4y_*,\frac{1}{2\eps}\right)}\ \sup_{x\in K\cap[0,\eps]}\ |\cos(2\pi y_nx) - \cos(2\pi y_*x)| \geqslant 1 + \frac{\sqrt{2}}{2},$$

\noindent
which concludes the proof. 
\end{proof}

\begin{thm}
The function $\hh_{\reals} : (\CC(\reals),\tau_{ucc}) \longrightarrow \reals_+\cup\{0\}$ given by the formula 
$$\hh_{\reals}(x\mapsto \cos(2\pi yx)) := y$$

\noindent
is a homeomorphism.
\label{hhrealsishomeo}
\end{thm}
\begin{proof}
It is easy to check that $\hh_{\reals}$ is a bijection, so we focus on the topological properties of this function. By Proposition 1.2 in \cite{Hu}, p. 152 the space $(\CC(\reals), \tau_{ucc})$ is second-countable and thus, by Theorem 1.6.14 in \cite{Engelking}, p. 53 it is sequential. This means that for $\hh_{\reals}$ to be a homeomorphism it is necessary and sufficient that $\hh_{\reals}^{-1}$ satisfies
$$y_n\longrightarrow y_* \ \Longleftrightarrow \ \hh_{\reals}^{-1}(y_n) \longrightarrow_{\tau_{ucc}} \hh_{\reals}^{-1}(y_*).$$

First, suppose that $(y_n)\subset \reals_+\cup\{0\}$ is a sequence convergent to $y_*.$ Then 
$$\forall_{x\in\reals}\ |\cos(2\pi y_n x) - \cos(2\pi y_* x)| = 2\pi|x| \left|\int_{y_*}^{y_n}\ \sin(2\pi yx)\ dy\right| \leqslant 2\pi |x| |y_n - y_*|,$$

\noindent
so for every compact $K\subset \reals$ we have
$$\sup_{x\in K}\ |\cos(2\pi y_n x) - \cos(2\pi y_* x)| \longrightarrow 0$$

\noindent
as $n\rightarrow \infty.$ Consequently, $\hh_{\reals}^{-1}(y_n) \longrightarrow_{\tau_{ucc}} \hh_{\reals}^{-1}(y_*)$ as desired. 

For the reverse implication (i.e., ``$\Longleftarrow$'') we suppose that $(y_n)\subset \reals_+\cup\{0\}$ is a sequence such that $\hh_{\reals}^{-1}(y_n) \longrightarrow_{\tau_{ucc}} \hh_{\reals}^{-1}(y_*)$ for some $y_* \in \reals_+\cup\{0\}.$ By Lemma \ref{lemmaonnonconvergence} the sequence $(y_n)$ is necessarily bounded, so using the Bolzano-Weierstrass theorem there exists a convergent subsequence $(y_{n_k}).$ If $y_{**}$ denotes the limit of this subsequence, then by the first part of the reasoning we have $\hh_{\reals}^{-1}(y_{n_k}) \longrightarrow_{\tau_{ucc}} \hh_{\reals}^{-1}(y_{**}).$ Since $\tau_{ucc}$ is a Hausdorff topology then it follows that $\hh_{\reals}^{-1}(y_{**}) = \hh_{\reals}^{-1}(y_*),$ which in turn implies the equality $y_* = y_{**}.$ Since the reasoning works for an arbitary choice of the subsequence we have $y_n \longrightarrow y_*,$ which concludes the proof.
\end{proof}

Let us prove a corresponding result for the cosine class $\CC(\integers):$
\begin{thm}
The function $\hh_{\integers} : (\CC(\integers),\tau_{ucc}) \longrightarrow S^1_+\cup\{1\}$ given by the formula 
$$\hh_{\integers}\left(k \mapsto \frac{z^k + z^{-k}}{2}\right) := z$$

\noindent
is a homeomorphism.
\label{hhintegersishomeo}
\end{thm}
\begin{proof}
Arguing as in Theorem \ref{hhrealsishomeo} it is enough to prove that 
$$z_n\longrightarrow z_* \ \Longleftrightarrow \ \hh_{\integers}^{-1}(z_n) \longrightarrow_{\tau_{ucc}} \hh_{\integers}^{-1}(z_*).$$

First, suppose that $(z_n)\subset S^1_+\cup\{1\}$ is a sequence convergent to $z_*.$ Then 
$$\forall_{k\in\integers}\ \left|\frac{z_n^k + z_n^{-k}}{2} - \frac{z_*^k+z_*^{-k}}{2}\right| \leqslant \frac{|z_n^k - z_*^k| + |z_n^{-k} - z_*^{-k}|}{2} \leqslant k\cdot \frac{|z_n - z_*| + |z_n^{-1} - z_*^{-1}|}{2} \stackrel{|z_n| = |z_*| = 1}{=} k |z_n - z_*|,$$

\noindent
so for every compact (i.e., finite) $K\subset \integers$ we have
$$\sup_{k\in K}\ \left|\frac{z_n^k + z_n^{-k}}{2} - \frac{z_*^k + z_*^{-k}}{2}\right| \longrightarrow 0$$

\noindent
as $n\rightarrow \infty.$ Consequently, $\hh_{\integers}^{-1}(z_n) \longrightarrow_{\tau_{ucc}} \hh_{\integers}^{-1}(z_*)$ as desired. 

For the reverse implication (i.e., ``$\Longleftarrow$'') we suppose that $(z_n)\subset S^1_+\cup\{1\}$ is a sequence such that $\hh_{\integers}^{-1}(z_n) \longrightarrow_{\tau_{ucc}} \hh_{\integers}^{-1}(z_*)$ for some $z_* = e^{2\pi i\alpha_*} \in S^1_+\cup\{1\}.$ Moreover, let $(\alpha_n)\subset [0,\frac{1}{2})$ be such that $z_n = e^{2\pi i\alpha_n}.$ Since the sequence $(\alpha_n)$ is bounded, then using the Bolzano-Weierstrass theorem there exists a convergent subsequence $(\alpha_{n_k}).$ If $\alpha_{**}$ denotes the limit of this subsequence, then $z_{n_k} = e^{2\pi i \alpha_{n_k}}\longrightarrow z_{**} = e^{2\pi i \alpha_{**}}$ and by the first part of the reasoning we have $\hh_{\integers}^{-1}(z_{n_k}) \longrightarrow_{\tau_{ucc}} \hh_{\integers}^{-1}(z_{**}).$ We conclude the proof just as in Theorem \ref{hhrealsishomeo}.
\end{proof}

In a similar vein (with even simpler proofs) one can prove analogous results for the remaining two cosine classes:
\begin{thm}
The functions 
\begin{equation*}
\begin{split}
\hh_{S^1} &: (\CC(S^1),\tau_{ucc}) \longrightarrow \naturals_0, \hspace{0.4cm} \hh_{S^1}(x\mapsto \cos(2\pi kx)) := k,\\
\hh_{\integers_n} &: (\CC(\integers_n),\tau_{ucc}) \longrightarrow \integers_{\lceil\frac{n+1}{2}\rceil}, \hspace{0.4cm} \hh_{\integers_n}\left(k\mapsto \cos\left(\frac{2\pi lk}{n}\right)\right) := l,\\
\end{split}
\end{equation*}

\noindent
are homeomorphisms.
\end{thm}

\section{Canonical cosine structure spaces and transforms}
\label{section:cosinestructurespaces}

In the previous section we have investigated the canonical cosine classes and found relatively simple spaces to which they are homeomorphic. We have also shown (see Theorem \ref{betaGisanopenmap}) that $\beta_G : \Delta(L^1(G),\tau^*)\longrightarrow (\CC(G),\tau_{ucc})$ is always an open map. This raises a natural question: can we compute the \textit{canonical cosine structure spaces} $\Delta(L^1(\reals),\star_c), \Delta(L^1(\integers),\star_c), \Delta(L^1(S^1),\star_c)$ and $\Delta(L^1(\integers_n),\star_c)$?\footnote{Again, by ``compute'' we mean ``find a topological space $T$, which is homeomorphic to $\Delta(L^1(G),\star_c)$.''} A major part of the present section is devoted to answering this question affirmatively.

\begin{thm}
$(\Delta(L^1(\reals), \star_c),\tau^*)$ is homeomorphic to $\reals_+\cup\{0\}.$
\label{deltal1reals}
\end{thm}
\begin{proof}
By Theorem \ref{betaGisanopenmap} we know that $\beta_{\reals} : (\Delta(L^1(\reals), \star_c),\tau^*) \longrightarrow (\CC(\reals),\tau_{ucc})$ is an open map and by Theorem \ref{hhrealsishomeo} the function $\hh_{\reals} : (\CC(\reals), \tau_{ucc}) \longrightarrow \reals_+\cup\{0\}$ is a homeomorphism. Consequently, it suffices to prove that $\hh_{\reals}\circ\beta_{\reals}$ is continuous. To this end we fix $y_*\in\reals_+\cup\{0\}$ as well as its arbitrary open neighbourhood 
$$U_{\eps} := \bigg\{y \in \reals_+\cup\{0\}\ :\ |y - y_*| < \eps\bigg\}$$

\noindent
where $\eps>0.$ Our task is to prove that 
$$(\hh_{\reals}\circ\beta_{\reals})^{-1}(U_{\eps}) = \bigg\{m \in \Delta(L^1(\reals),\star_c)\ :\ |\hh_{\reals}\circ\beta_{\reals}(m) - y_*| < \eps\bigg\}$$

\noindent
is weak* open and we do it by fixing an arbitrary element $m_{**}\in (\hh_{\reals}\circ\beta_{\reals})^{-1}(U_{\eps})$ and constructing a weak* open set $W_{**}$ such that 
$$m_{**}\in W_{**}\subset (\hh_{\reals}\circ\beta_{\reals})^{-1}(U_{\eps}).$$ 

To begin with, since $m_{**}\in (\hh_{\reals}\circ\beta_{\reals})^{-1}(U_{\eps})$ then $y_{**}:= \hh_{\reals}\circ\beta_{\reals}(m_{**})$ satisfies $|y_{**} - y_*|<\delta\eps$ for some $\delta \in [0,1).$ Further reasoning depends on whether $y_{**}$ is zero or not:
\begin{itemize}
	\item If $y_{**}\neq 0$ then we define a function $g:\reals_+\cup\{0\}\longrightarrow\reals$ with the formula
	$$g(z) := \frac{1}{2\pi y_{**}}\cdot\left(\frac{\pi}{2}\cdot \sinc\left(\frac{\pi}{2}\cdot z\right) - 1\right).$$
	
	\noindent
	Its crucial property is that there exists $\eta>0$ such that 
	\begin{gather}
	\forall_{z\in\reals_+\cup\{0\}}\ |g(z)| < \eta \ \Longrightarrow \ |z-1| < \frac{(1-\delta)\eps}{y_{**}}.
	\label{inverseofsinc}
	\end{gather}
	
	We claim that 
	$$W_{**}:= \bigg\{m \in \Delta(L^1(\reals),\star_c)\ :\ \left|m\left(\mathds{1}_{[0,\frac{1}{4y_{**}}]}\right) - m_{**}\left(\mathds{1}_{[0,\frac{1}{4y_{**}}]}\right)\right| < \eta\bigg\}$$
	
	\noindent
	is the desired weak* open neighbourhood of $m_{**}.$ Indeed, we have 
	\begin{equation*}
	\begin{split}
	\forall_{m\in\Delta(L^1(\reals),\star_c)}\ m\left(\mathds{1}_{[0,\frac{1}{4y_{**}}]}\right) - m_{**}\left(\mathds{1}_{[0,\frac{1}{4y_{**}}]}\right) &= \int_0^{\frac{1}{4y_{**}}}\ \cos(2\pi yx) - \cos(2\pi y_{**}x)\ dx \\
	&\stackrel{x\mapsto \frac{x}{2\pi y_{**}}}{=}  \frac{1}{2\pi y_{**}}\cdot \int_0^{\frac{\pi}{2}}\ \cos\left(\frac{y}{y_{**}}\cdot x\right) - \cos(x)\ dx \\
	&= \frac{1}{2\pi y_{**}}\cdot \left(\frac{\pi}{2}\cdot \sinc\left(\frac{\pi}{2}\cdot \frac{y}{y_{**}}\right)-1\right) = g\left(\frac{y}{y_{**}}\right),
	\end{split}
	\end{equation*}
	
	\noindent
	where $y = \hh_{\reals}\circ\beta_{\reals}(m).$ If $m\in W_{**}$ then $\left|g\left(\frac{y}{y_{**}}\right)\right| < \eta,$ so by \eqref{inverseofsinc} we have $|y-y_{**}| < (1-\delta)\eps.$ Finally, we have 
	$$\forall_{m\in W_{**}}\ |y - y_*| \leqslant |y - y_{**}| + |y_{**} - y_*| < \delta \eps + (1-\delta)\eps = \eps,$$
	
	\noindent
	which proves that $m_{**}\in W_{**} \subset (\hh_{\reals}\circ\beta_{\reals})^{-1}(U_{\eps}).$
	
	\item If $y_{**} = 0$ then we define a function $g:\reals_+\cup\{0\}\longrightarrow\reals$ with the formula $g(z) := \sinc(2\pi z) - 1.$ Its crucial property is that there exists $\eta>0$ such that 
	\begin{gather}
	\forall_{z\in\reals_+\cup\{0\}}\ |g(z)| < \eta \ \Longrightarrow \ |z| < (1-\delta)\eps.
	\label{inverseofsinccase0}
	\end{gather}
	
	We claim that 
	$$W_{**}:= \bigg\{m \in \Delta(L^1(\reals),\star_c)\ :\ \left|m\left(\mathds{1}_{[0,1]}\right) - m_{**}\left(\mathds{1}_{[0,1]}\right)\right| < \eta\bigg\}$$
	
	\noindent
	is the desired weak* open neighbourhood of $m_{**}.$ Indeed, we have 
	\begin{gather*}
	\forall_{m\in\Delta(L^1(\reals),\star_c)}\ m\left(\mathds{1}_{[0,1]}\right) - m_{**}\left(\mathds{1}_{[0,1]}\right) = \int_0^1\ \cos(2\pi yx) - 1\ dx = \sinc(2\pi y) - 1 = g(y).
	\end{gather*}
	
	\noindent
	If $m\in W_{**}$ then $|g(y)| < \eta,$ so by \eqref{inverseofsinccase0} we have $|y| < (1-\delta)\eps.$ We conclude the reasoning as in the previous case.
\end{itemize}
\end{proof}

\begin{thm}
$(\Delta(\ell^1(\integers),\star_c),\tau^*)$ is homeomorphic to $S^1_+\cup\{1\}.$
\label{deltaell1Z}
\end{thm}
\begin{proof}
As in Theorem \ref{deltal1reals} we argue that it is sufficient to prove that $\hh_{\integers}\circ\beta_{\integers}$ is continuous so we choose an arbitrary open neighbourhood  
$$U_{\eps} := \bigg\{z \in S^1_+\cup\{1\}\ :\ |z - z_*| < \eps\bigg\}$$

\noindent
of a fixed element $z_* \in S^1_+\cup\{1\}.$ Our task is to prove that 
$$(\hh_{\integers}\circ\beta_{\integers})^{-1}(U_{\eps}) = \bigg\{m \in \Delta(\ell^1(\integers),\star_c)\ :\ |\hh_{\integers}\circ\beta_{\integers}(m) - z_*| < \eps\bigg\}$$

\noindent
is weak* open and we do it by fixing an arbitrary element $m_{**}\in (\hh_{\integers}\circ\beta_{\integers})^{-1}(U_{\eps})$ and constructing a weak* open set $W_{**}$ such that 
$$m_{**}\in W_{**}\subset (\hh_{\integers}\circ\beta_{\integers})^{-1}(U_{\eps}).$$ 

To begin with, since $m_{**}\in (\hh_{\integers}\circ\beta_{\integers})^{-1}(U_{\eps})$ then $z_{**}:= \hh_{\integers}\circ\beta_{\integers}(m_{**})$ satisfies $|z_{**} - z_*| < \delta\eps$ for some $\delta \in (0,1).$ We define a function $g:S^1_+\cup\{1\} \longrightarrow\reals$ with the formula
$$g(z) := \frac{z + z^{-1}}{2}  - \frac{z_{**} + z_{**}^{-1}}{2} .$$
	
\noindent
Its crucial property is that there exists $\eta>0$ such that 
\begin{gather}
\forall_{z\in S^1_+\cup\{1\}}\ |g(z)| < \eta \ \Longrightarrow \ |z-z_{**}| < (1-\delta)\eps.
\label{inverseofcos}
\end{gather}
	
We claim that 
$$W_{**}:= \bigg\{m \in \Delta(\ell^1(\integers),\star_c)\ :\ \left|m\left(\mathds{1}_{\{1\}}\right) - m_{**}\left(\mathds{1}_{\{1\}}\right)\right| < \eta\bigg\}$$
	
\noindent
is the desired weak* open neighbourhood of $m_{**}.$ Indeed, we have 
\begin{equation*}
\begin{split}
\forall_{m\in\Delta(\ell^1(\integers),\star_c)}\ m\left(\mathds{1}_{\{1\}}\right) - m_{**}\left(\mathds{1}_{\{1\}}\right)  &= \sum_{k\in\integers}\ \mathds{1}_{\{1\}}(k)\cdot \bigg(\frac{z^k + z^{-k}}{2} - \frac{z_{**}^k + z_{**}^{-k}}{2}\bigg)\\
&= \frac{z + z^{-1}}{2} - \frac{z_{**} + z_{**}^{-1}}{2} = g(z),
\end{split}
\end{equation*}
	
\noindent
where $z = \hh_{\integers}\circ\beta_{\integers}(m).$ If $m\in W_{**}$ then $|g(z)| < \eta,$ so by \eqref{inverseofcos} we have $|z-z_{**}| < (1-\delta)\eps.$ Finally, we have 
$$\forall_{m\in W_{**}}\ |z - z_*| \leqslant |z - z_{**}| + |z_{**} - z_*| < \delta \eps + (1-\delta)\eps = \eps,$$
	
\noindent
which proves that $m_{**}\in W_{**} \subset (\hh_{\integers}\circ\beta_{\integers})^{-1}(U_{\eps}).$ We conclude the reasoning as in Theorem \ref{deltal1reals}.	
\end{proof}

\begin{thm}
$(\Delta(L^1(S^1), \star_c),\tau^*)$ is homeomorphic to $\naturals_0.$
\label{deltal1realsslashintegers}
\end{thm}
\begin{proof}
As in Theorem \ref{deltal1reals} we argue that it is sufficient to prove that $\hh_{S^1}\circ\beta_{S^1}$ is continuous. Since the topology on $\naturals_0$ is discrete, then we have to show that 
$$\{m_*\} := (\hh_{S^1}\circ\beta_{S^1})^{-1}(\{k_*\})$$

\noindent
is weak* open for every $k_*\in \naturals_0.$ Further reasoning depends on whether $k_*$ is zero or not:

\begin{itemize}
	\item If $k_*\neq 0$ then we define a function $g:\naturals_0\longrightarrow\reals$ with the formula
	$$g(k) := \frac{1}{2\pi k_*}\cdot\left(\frac{\pi}{2}\cdot \sinc\left(\frac{\pi}{2}\cdot \frac{k}{k_*}\right) - 1\right).$$
	
	\noindent
	Its crucial property is that there exists $\eta>0$ such that 
	\begin{gather}
	\forall_{k\in\naturals_0}\ |g(k)| < \eta \ \Longrightarrow \ k = k_*.
	\label{inverseofsincdiscrete}
	\end{gather}
	
	We claim that 
	$$W_* := \bigg\{m \in \Delta(L^1(S^1),\star_c)\ :\ \left|m\left(\mathds{1}_{[0,\frac{1}{4k_*}]}\right) - m_*\left(\mathds{1}_{[0,\frac{1}{4k_*}]}\right)\right| < \eta\bigg\}$$
	
	\noindent
	is the desired weak* open neighbourhood of $m_*.$ Indeed, we have 
	\begin{equation*}
	\begin{split}
	\forall_{m\in\Delta(L^1(S^1),\star_c)}\ m\left(\mathds{1}_{[0,\frac{1}{4k_*}]}\right) - m_*\left(\mathds{1}_{[0,\frac{1}{4k_*}]}\right) &= \int_0^{\frac{1}{4k_*}}\ \cos(2\pi kx) - \cos(2\pi k_*x)\ dx \\
	&\stackrel{x\mapsto \frac{x}{2\pi k_*}}{=}  \frac{1}{2\pi k_*}\cdot \int_0^{\frac{\pi}{2}}\ \cos\left(\frac{k}{k_{**}}\cdot x\right) - \cos(x)\ dx \\
	&= \frac{1}{2\pi k_*}\cdot \left(\frac{\pi}{2}\cdot \sinc\left(\frac{\pi}{2}\cdot \frac{k}{k_*}\right)-1\right) = g(k),
	\end{split}
	\end{equation*}
	
	\noindent
	where $k = \hh_{S^1}\circ\beta_{S^1}(m).$ If $m\in W_*$ then $|g(k)| < \eta,$ so by \eqref{inverseofsincdiscrete} we have $k = k_*.$  This proves that 
	$$\{m_*\} = W_* = (\hh_{S^1}\circ\beta_{S^1})^{-1}(\{k_*\}).$$
	
	\item If $k_* = 0$ then we define a function $g:\naturals_0\longrightarrow\reals$ with the formula $g(k) := \sinc(2\pi k) - 1.$ Its crucial property is that there exists $\eta>0$ such that 
	\begin{gather}
	\forall_{k\in\naturals_0}\ |g(k)| < \eta \ \Longrightarrow \ k = 0.
	\label{inverseofsinccase0discrete}
	\end{gather}
	
	We claim that 
	$$W_*:= \bigg\{m \in \Delta(L^1(S^1),\star_c)\ :\ \left|m\left(\mathds{1}_{S^1}\right) - m_*\left(\mathds{1}_{S^1}\right)\right| < \eta\bigg\}$$
	
	\noindent
	is the desired weak* open neighbourhood of $m_*.$ Indeed, we have 
	\begin{gather*}
	\forall_{m\in\Delta(L^1(S^1),\star_c)}\ m\left(\mathds{1}_{S^1}\right) - m_*\left(\mathds{1}_{S^1}\right) = \int_0^1\ \cos(2\pi kx) - 1\ dx = \sinc(2\pi k) - 1 = g(k).
	\end{gather*}
	
	\noindent
	If $m\in W_*$ then $|g(k)| < \eta,$ so by \eqref{inverseofsinccase0discrete} we have $k = 0.$ This proves that 
	$$\{m_*\} = W_* = (\hh_{S^1}\circ\beta_{S^1})^{-1}(\{0\}),$$
	
	\noindent
	which concludes the proof.
\end{itemize}
\end{proof}

\begin{thm}
$(\Delta(\ell^1(\integers_n),\star_c),\tau^*)$ is homeomorphic to $\integers_{\lceil\frac{n+1}{2}\rceil},$ where $x\mapsto \lceil x\rceil$ is the ceiling function.
\label{deltaell1}
\end{thm}
\begin{proof}
As in Theorem \ref{deltal1realsslashintegers} our task is to prove that
$$\{m_*\} := (\hh_{\integers_n}\circ\beta_{\integers_n})^{-1}(\{k_*\})$$

\noindent
is a weak* open neighbourhood of $m_*$ for every $k_* \in \integers_n.$ We define a function $g:\integers_{\lceil\frac{n+1}{2}\rceil} \longrightarrow\reals$ with the formula
$$g(k) := \cos\left(\frac{2\pi k}{n}\right) - \cos\left(\frac{2\pi k_*}{n}\right).$$
	
\noindent
Its crucial property is that there exists $\eta>0$ such that 
\begin{gather}
\forall_{k\in\integers_n}\ |g(k)| < \eta \ \Longrightarrow \ k=k_*.
\label{inverseofcosintegersn}
\end{gather}
	
We claim that 
$$W_*:= \bigg\{m \in \Delta(\ell^1(\integers_n),\star_c)\ :\ \left|m\left(\mathds{1}_{\{1\}}\right) - m_*\left(\mathds{1}_{\{1\}}\right)\right| < \eta\bigg\}$$
	
\noindent
is the desired weak* open neighbourhood of $m_*.$ Indeed, we have 
\begin{equation*}
\forall_{m\in\Delta(\ell^1(\integers_n),\star_c)}\ m\left(\mathds{1}_{\{1\}}\right) - m_*\left(\mathds{1}_{\{1\}}\right)  = \cos\left(\frac{2\pi k}{n}\right) - \cos\left(\frac{2\pi k_*}{n}\right) = g(k),
\end{equation*}
	
\noindent
where $k = \hh_{\integers_n}\circ\beta_{\integers_n}(m).$ If $m\in W_*$ then $|g(k)| < \eta,$ so by \eqref{inverseofcosintegersn} we have $k=k_*.$ This proves that 
$$\{m_*\} = W_* = (\hh_{\integers_n}\circ\beta_{\integers_n})^{-1}(\{k_*\}),$$

\noindent
which concludes the proof.
\end{proof}

The last four theorems can be summarized as follows:

\begin{thm}
The cosine structure space $\Delta(L^1(G),\star_c)$ is homeomorphic to:
\begin{itemize}
	\item $\reals_+\cup\{0\}$ if $G = \reals,$
	\item $S^1_+\cup\{1\}$ if $G = \integers,$
	\item $\naturals_0$ if $G = S^1,$
	\item $\integers_{\lceil \frac{n+1}{2}\rceil}$ if $G = \integers_n.$
\end{itemize}
\label{fourhomeomorphisms}
\end{thm}

It is high time we reaped what we have sown and enjoyed the fruits of our labour. Due to Theorem \ref{fourhomeomorphisms} we know that 
\begin{itemize}
	\item $C_0(\Delta(L^1(\reals),\star_c))$ is homeomorphic to $C_0(\reals_+\cup\{0\}),$
	\item $C_0(\Delta(\ell^1(\integers),\star_c))$ is homeomorphic to $C_0(S^1_+\cup\{1\}),$
	\item $C_0(\Delta(L^1(S^1),\star_c))$ is homeomorphic to $C_0(\naturals_0),$
	\item $C_0(\Delta(\ell^1(\integers_n),\star_c))$ is homeomorphic to $C_0(\integers_{\lceil \frac{n+1}{2}\rceil}) = C(\integers_{\lceil \frac{n+1}{2}\rceil}) = \complex^{\lceil \frac{n+1}{2}\rceil}.$
\end{itemize}

\noindent
Hence, the Gelfand transform $\widehat{f} \in C_0(\Delta(L^1(G),\star_c))$ of a function $f\in L^1(G)$ manifests itself as
\begin{itemize}
	\item the \textit{classical cosine transform} 
	$$\forall_{y\in \reals_+\cup\{0\}}\ \widehat{f}(y) = \int_{\reals}\ f(x)\cos(2\pi yx)\ dx,$$
	
	\noindent
	if $G = \reals,$
	
	\item the \textit{discrete-time cosine transform} 
	$$\forall_{z \in S^1_+\cup\{1\}}\ \widehat{f}(z) = \sum_{k\in\integers}\ f(k)\cdot \frac{z^k + z^{-k}}{2},$$
	
	\noindent
	if $G = \integers,$
	
	\item the $k-$th cosine coefficient in the Fourier series
	$$\forall_{k\in \naturals_0}\ \widehat{f}(k) = \int_0^1\ f(x)\cos(2\pi kx)\ dx,$$
	
	\noindent
	if $G = S^1,$ 
	
	\item the \textit{discrete cosine transform}
	$$\forall_{l\in \integers_{\lceil \frac{n+1}{2}\rceil}}\ \widehat{f}(l) = \sum_{k=1}^n f(k)\cos\left(\frac{2\pi lk}{n}\right),$$
	
	\noindent
	if $G = \integers_n.$ 
\end{itemize}

\section*{Epilogue}

Our journey has come to an end and it is instructive to pause one last time and, with the benefit of hindsight, reflect on how far we have travelled and what lies ahead. Last section taught us that $\Delta(L^1(G),\star_c)$ is (homeomorphic to) a relatively simple topological space if $G = \reals,\integers,S^1$ or $\integers_n$. As a result, we rediscovered the cosine transforms as special manifestations of the Gelfand transform. However, one would be wrong thinking that the topic has been exhausted. Four homeomorphisms, which appear in Theorem \ref{fourhomeomorphisms}, force all four functions $\beta_{\reals}, \beta_{\integers},\beta_{S^1}$ and $\beta_{\integers_n}$ to be homeomorphisms as well. This raises the very natural question: is it true that $\beta_G:\Delta(L^1(G),\star_c)\longrightarrow \CC(G)$ is a homeomorphism for every locally compact abelian group $G$? Unfortunately, despite our best efforts we were not able to answer that question. Thus, we leave it as an open problem with the intention of stimulating future research in the fascinating field of cosine transforms.


\begin{thebibliography}{9}
	\bibitem{AguiAraiNakajima}
		Agui T., Arai Y., Nakajima M. : \textit{A fast DCT-SQ scheme for images}, Trans. IEICE, Vol. 71 (11), p. 1095-1097 (1988)
	\bibitem{Ahmeddiary}
		Ahmed N. : \textit{How I Came Up with the Discrete Cosine Transform}, Digit. Signal Process., Vol. 1, p. 4-5 (1991)
	\bibitem{AhmedMagotraMandyam}
		Ahmed N., Magotra N., Mandyam G. : \textit{Lossless Image Compression Using the Discrete Cosine Transform}, J. Vis. Commun. Image R., Vol. 8 (1), p. 21-26 (1997) 
	\bibitem{AhmedNatarajanRao}
		Ahmed N., Natarajan T., Rao R. K. : \textit{Discrete cosine transform}, IEEE T. Comput., Vol. C-23 (1), p. 90-93 (1974)
	\bibitem{AkopiaAstolaSaarinenTakala}
		Akopian D., Astola J., Saarinen J., Takala J. : \textit{Constant geometry algorithm for discrete cosine transform}, IEEE Trans. Signal Process., Vol. 48 (6), p. 1840-1843 (2000)
	\bibitem{AliprantisBorder}
		Aliprantis C. D., Border K. C. : \textit{Infinite Dimensional Analysis. A Hitchhiker's Guide}, Springer-Verlag, Berlin, 2006 
	\bibitem{ArguelloZapata}
		Arguello F., Zapata E. L. : \textit{Fast cosine transform based on the successive doubling method}, Electronic Lett., Vol. 26 (19), p. 1616-1618
(1990)
	\bibitem{Bobrowski}
		Bobrowski A. : \textit{Functional Analysis for Probability and Stochastic Processes}, Cambridge University Press, Cambridge, 2005 
	\bibitem{Brezis}
		Brezis H. : \textit{Functional Analysis, Sobolev Spaces and Partial Differential Equations}, Springer-Verlag, New York, 2010
	\bibitem{Cartan}
		Cartan, H. : \textit{Sur la mesure de Haar}, C. R. Math. Acad. Sci. Paris, Vol. 211, p. 759-762 (1940)
	\bibitem{ChenFralickSmith}
		Chen W. H., Fralick S. C., Smith C. H. : \textit{A fast computational algorithm for the discrete cosine transform}, IEEE Trans. Comm., Vol. 25 (9), p. 1004-1009 (1977)
	\bibitem{Deitmar}
		Deitmar A. : \textit{A First Course in Harmonic Analysis}, Springer-Verlag, New York, 2005
	\bibitem{DeitmarEchterhoff}
	  Deitmar A., Echterhoff S. : \textit{Principles of Harmonic Analysis}, Springer, New York, 2009
	\bibitem{DiestelSpalsbury}
		Diestel J., Spalsbury A. : \textit{The Joys of Haar measure}, American Mathematical Society, Providence, 2014
	\bibitem{Dinculeanu}
		Dinculeanu N. : \textit{Vector measures}, Pergamon Press, Berlin, 1967
	\bibitem{Engelking}
		Engelking R. : \textit{General Topology}, Heldermann Verlag, Berlin, 1989
	\bibitem{Folland}
		Folland G. B. : \textit{A Course in Abstract Harmonic Analysis}, CRC Press, Boca Raton, 1995
	\bibitem{Fourier}
		Fourier J. B. : \textit{The analytical theory of heat}, Cambridge University Press, Cambridge, 2009
	\bibitem{GohbergGoldbergKaashoek}
		Gohberg I., Goldberg S., Kaashoek M. A. : \textit{Classes of Linear Operators, Vol. 1}, Springer, Basel, 1990
	\bibitem{Grillet}
		Grillet P. A. : \textit{Abstract Algebra. Second Edition}, Springer-Verlag, New York, 2007
	\bibitem{GuoShiWang}
		Guo Z., Shi B., Wang N. : \textit{Two new algorithms based on product system for discrete cosine transform}, Signal Process., Vol. 81, p. 1899-1908 (2001)
	\bibitem{Haar}
		Haar, A. : \textit{Der Massbegriff in der Theorie der kontinuierlichen Gruppen}, Ann. Math., Vol. 34 (1), p. 147-169 (1933)
	\bibitem{Hoffman}
		Hoffman R. : \textit{Data Compression in Digital Systems}, Springer Science+Business Media Dordrecht, 1997
	\bibitem{Hou}
		Hou H. S. : \textit{A fast algorithm for computing the discrete cosine transform}, IEEE Trans. Acoust. Speech Signal Process., Vol. 35 (10), p. 1455-1461 (1987)
	\bibitem{Hu}
		Hu Sz.-T. : \textit{Elements of general topology}, Holden-Day, San Francisco, 1969 
	\bibitem{JohnsonShao}
		Johnson S. G., Shao X. : \textit{Type-II/III DCT/DST algorithms with reduced number of arithmetic operations}, Signal Process., Vol. 88 (6), p. 1553-1564 (2008) 
	\bibitem{Kaniuth}
		Kaniuth E. : \textit{A Course in Commutative Banach Algebras}, Springer-Verlag, New York, 2009
	\bibitem{Kannappan}
		Kannappan P. : \textit{The functional equation $f(xy) + f(xy^{-1}) = 2f(x)f(y)$ for groups}, Proc. Am. Math. Soc., Vol. 19, p. 69-74 (1968)
	\bibitem{Kelley}
		Kelley J. L. : \textit{General Topology}, Springer-Verlag, New York, 1975
	\bibitem{Kok}
		Kok C. W. : \textit{Fast algorithm for computing discrete cosine transform}, IEEE Trans. Signal Process., Vol. 45 (3), p. 757-760 (1997)
	\bibitem{Lee}
		Lee B. G. : \textit{A new algorithm to compute the discrete cosine transform}, IEEE Trans. Acoust. Speech Signal Process., Vol. 32 (6), p. 1243-1245 (1984)
	\bibitem{Li}
		Li W. : \textit{A new algorithm to compute the DCT and its inverse}, IEEE Trans. Signal Process., Vol. 39 (6), p. 1305-1313 (1991)
	\bibitem{MillerTseng}
		Miller W. C., Tseng B. D. : \textit{On computing the discrete cosine transform}, IEEE Trans. Comput., Vol. 27 (10), p. 966-968 (1978)
	\bibitem{Munkres}
		Munkres J. R. : \textit{Topology}, Prentice Hall, Upper Saddle River, 2000
	\bibitem{NarasimhaPeterson}
		Narasimha M. J., Peterson A. M. : \textit{On the computation of the discrete cosine transform}, IEEE Trans. Comm., Vol. 26 (6), p. 934-936 (1978)
	\bibitem{NussbaumerVetterli}
		Nussbaumer H. J., Vetterli M. : \textit{Simple FFT and DCT algorithms with reduced number of operations}, Signal Process., Vol. 6 (4), p. 267-278 (1984)
	\bibitem{Pedersen}
		 Pedersen G. K. : \textit{Analysis Now}, Springer-Verlag, New York, 1989
	\bibitem{PlonkaTasche}
		Plonka G., Tasche M. : \textit{Fast and numerically stable algorithms for discrete cosine transforms}, Linear Algebra Appl., Vol. 394, p. 309-345 (2005)
	\bibitem{Terras}
		Terras A. : \textit{Fourier Analysis on Finite Groups and Applications}, Cambridge University Press, Cambridge, 1999
	\bibitem{Wallace}
		 Wallace G. K. : \textit{The  JPEG  still  picture  compression standard}, Commun. ACM, Vol. 34 (4), p. 30-44 (1991)
	\bibitem{Wang}
		Wang Z. : \textit{A fast algorithm for the discrete sine transform implemented by the fast cosine transform}, IEEE Trans. Acoust. Speech Signal Process., Vol. 30 (5), p. 813-815 (1982)
	\bibitem{Weil}
		Weil A. : \textit{L'int\'{e}gration dans les groupes topologiques et ses applications}, Hermann, Paris, 1965 
	\bibitem{Willard}
		Willard S. : \textit{General topology}, Addison-Wesley Publishing Company, Reading, 1970
\end{thebibliography}
\end{document}